\documentclass[leqno,12pt]{article} 
\usepackage{amsmath, amssymb}
\usepackage{amsthm} 
\theoremstyle{plain} 
\newtheorem{theorem}{\indent\sc Theorem}[section]
\newtheorem{lemma}[theorem]{\indent\sc Lemma}
\newtheorem{corollary}[theorem]{\indent\sc Corollary}
\newtheorem{proposition}[theorem]{\indent\sc Proposition}

\theoremstyle{definition} 
\newtheorem{definition}[theorem]{\indent\sc Definition}
\newtheorem{remark}[theorem]{\indent\sc Remark}
\newtheorem{example}[theorem]{\indent\sc Example}


\title{Invariant, anti-invariant and slant submanifolds \\of a metallic Riemannian manifold} 
\author{Adara M. Blaga and Cristina E. Hre\c tcanu}
\date{}

\begin{document}
\maketitle

\markboth{{\small\it {\hspace{4cm} Slant submanifolds of a metallic Riemannian manifold}}}{\small\it{Slant submanifolds of a metallic Riemannian manifold
\hspace{4cm}}}


\begin{abstract}
Properties of invariant, anti-invariant and slant isometrically immersed submanifolds of metallic Riemannian manifolds are given with a special view towards the induced $\Sigma$-structure. Examples of such metallic manifolds are also given.
\\[2mm] {\it AMS Mathematics  Subject Classification $(2010)$}: 53C07, 53C25, 53C40, 53C42
\\[1mm] {\it Key words and phrases:} invariant, anti-invariant and slant submanifold; metallic Riemannian structure
\end{abstract}

\section{Introduction}

The theory of submanifolds has the origin in the study of the geometry of plane curves initiated by Fermat. Since then it has been evolving in different directions of differential geometry and mechanics, especially. It is still an active and vast research field playing an important role in the development of modern differential geometry. The modeling spaces of dynamical systems always carry different canonical geometrical objects: affine connections, differential forms, tensor fields etc. A natural question arising is when the submanifold inherits the geometrical structures of the ambient manifold. In this spirit, we shall consider a certain kind of isometrically immersed submanifolds of metallic Riemannian manifolds, namely, \textit{slant} submanifolds. First time, the notion of slant submanifold appeared for complex manifolds in Chen's book \cite{chen2}. Remark that the slant submanifolds have been studied in different other context: for contact \cite{lott1}, LP-contact \cite{khan2}, K-contact \cite{lott2}, K\"{a}hler \cite{chen1}, Sasakian \cite{cabr}, Lorentzian \cite{per}, Kenmotsu \cite{gup}, para-Kenmotsu \cite{blag}, almost product Riemannian manifolds \cite{sah}, almost paracontact metric manifolds \cite{we} etc.

We shall begin recalling the basic properties of a metallic Riemannian structure and prove some immediate consequences of the Gauss and Weingarten equations for an isometrically immersed submanifold in a metallic Riemannian manifold $(\bar{M},J, g)$.  We also consider the $\Sigma$-structure induced by a metallic Riemannian structure on its submanifolds and establish a kind of inheritance property to the submanifolds of isometrically immersed submanifolds of $\bar{M}$. In the main section we characterize the invariant, anti-invariant and slant submanifolds of $\bar{M}$.

\section{Metallic Riemannian manifolds revisited}

\begin{definition} \cite{c}
A $(1,1)$-tensor field $J$ is called \textit{metallic structure} on $M$ if it satisfies the equation:
\begin{equation}
J^2=pJ+qI_{\Gamma(TM)},
\end{equation}
for $p$, $q\in\mathbb{N}^*$, where $I_{\Gamma(TM)}$ is the identity operator on $\Gamma(TM)$. The pair $(M, J)$ is a {\it metallic manifold}. Moreover, if a Riemannian metric $g$ on $M$ is compatible with $J$, that is $g(JX, Y)=g(X, JY)$, for any $X$, $Y\in \Gamma(TM)$, we call the pair $(J,g)$ {\it metallic Riemannian structure} and $(M, J,g)$ {\it metallic Riemannian manifold}.
\end{definition}

It was shown \cite{c} that the powers of $J$ satisfy:
\begin{equation}
J^n=g_{n}J+qg_{n-1}I_{\Gamma(TM)},
\end{equation}
where $\{g_n\}_{n\in\mathbb{N}^*}$ is the generalized secondary Fibonacci sequence defined by $g_{n+1}=pg_n+qg_{n-1}$, $n\geq 1$ with $g_0=0$, $g_1=1$ and $p$, $q$ real numbers.

\begin{remark}
Concerning the inheriting of this kind of structure on submanifolds, Hre\c tcanu and Crasmareanu proved in \cite{c} that a metallic structure on a metallic Riemannian manifold $M$ induced a metallic structure on every invariant submanifold of $M$ and illustrate this on a product of spheres in an Euclidian space.
\end{remark}

Fix now $J$ a metallic structure on $M$ and define the associated linear connections as follows:

\begin{definition}

i) A linear connection $\nabla$ on $M$ is called $J$-{\it connection} if $J$ is covariant constant with respect to $\nabla$,
namely $\nabla J=0$.

ii) If the Levi-Civita connection $\nabla$ with respect to a Riemannian metric $g$ on $M$ compatible with $J$ is $J$-connection, then $(M,J,g)$ is called \textit{locally metallic Riemannian manifold}.
\end{definition}

The concept of integrability is defined in the classical manner:

\begin{definition}
A metallic structure $J$ is called {\it integrable} if its Nijenhuis tensor field
$N_{J}(X, Y):=[JX, JY]-J[JX, Y]-J[X, JY] +J^{2}[X, Y]$ vanishes.
\end{definition}

Necessary and sufficient conditions for the integrability of a polynomial structure $J$ whose characteristic polynomial has only simple roots were given by Vanzura in \cite{v} who proved that if there exists a symmetric linear $J$-connection $\nabla$, then the structure $J$ is integrable.

\section{Submanifolds of metallic Riemannian manifolds}

\subsection{Isometrically immersed submanifolds}

We shall focus on a certain kind of isometrically immersed submanifolds of metallic Riemannian manifolds, namely, \textit{slant} submanifolds.

Let $M$ be an $n$-dimensional submanifold of codimension $r$ isometrically immersed in an $(n+r)$-dimensional metallic Riemannian manifold $(\bar{M},J,g)$ ($n$, $r\in\mathbb{N}^*$). Then for each $x\in M$, the tangent space $T_x\bar{M}$ of $\bar{M}$ decomposes into the direct sum:
$$T_x\bar{M}=T_xM\oplus T_xM^{\perp}.$$

Denote by:
\begin{equation}
T:\Gamma(TM)\rightarrow \Gamma(TM), \ \ TX:=(J X)^T,
\end{equation}
\begin{equation}
N:\Gamma(TM)\rightarrow \Gamma(TM^{\perp}), \ \ NX:=(J X)^{\perp},
\end{equation}
\begin{equation}
t:\Gamma(TM^{\perp})\rightarrow \Gamma(TM), \ \ tU:=(J U)^T,
\end{equation}
\begin{equation}
n:\Gamma(TM^{\perp})\rightarrow \Gamma(TM^{\perp}), \ \ nU:=(J U)^{\perp}.
\end{equation}

Remark that the maps $T$ and $n$ are $g$-symmetric:
\begin{equation}
g(TX,Y)=g(X,TY), \ \ X, Y\in \Gamma(TM),
\end{equation}
\begin{equation}
g(nU,V)=g(U,nV), \ \ U, V\in \Gamma(TM^{\perp})
\end{equation}
and
\begin{equation}
g(NX,U)=g(X,tU), \ \ X\in \Gamma(TM), U\in \Gamma(TM^{\perp}).
\end{equation}

Denoting also by $g$ the Riemannian metric induced on $M$, by $\bar{\nabla}$ and $\nabla$ the Levi-Civita connections on $(\bar{M},g)$ and $(M,g)$ respectively and by $\{N_1,\dots,N_r\}$ an orthonormal basis for the normal space, the Gauss and Weingarten formulas corresponding to $M$ are given by:
\begin{equation}
\bar{\nabla}_{X}Y=\nabla_XY+\sum_{\alpha=1}^rh_{\alpha}(X,Y)N_{\alpha},
\end{equation}
\begin{equation}\label{e5}
\bar{\nabla}_{X}N_{\alpha}=-A_{N_{\alpha}}X+\sum_{\beta=1}^r\lambda_{\alpha\beta}(X)N_{\beta},
\end{equation}
where $h_{\alpha}$, $1\leq \alpha \leq r$, are the (symmetric) second fundamental tensors corresponding to $N_{\alpha}$, i.e. $h(X,Y)=\sum_{\alpha=1}^rh_{\alpha}(X,Y)N_{\alpha}$, for $X$, $Y\in\Gamma(TM)$, $A_{N_{\alpha}}$ is the shape operator (or the Weingarten map) in the direction of the normal vector field $N_{\alpha}$ defined by $g(A_{N_{\alpha}}X,Y)=h_{\alpha}(X,Y)$, for $X$, $Y\in\Gamma(TM)$, $1\leq \alpha \leq r$ and $\lambda_{\alpha\beta}=-\lambda_{\beta\alpha}$, $1\leq \alpha,\beta \leq r$, the $1$-forms on $M$ corresponding to the normal connection $\nabla^{\perp}$, i.e. $\nabla^{\perp}_XN_{\alpha}=\sum_{\beta=1}^r\lambda_{\alpha\beta}(X)N_{\beta}$.

Also from $g(JX,Y)=g(X,JY)$ follows $g((\bar{\nabla}_XJ)Y,Z)=g(Y,(\bar{\nabla}_XJ)Z)$, for any $X$, $Y$, $Z\in \Gamma(T\bar{M})$.

\begin{proposition}
If $M$ is an isometrically immersed submanifold of the metallic Riemannian manifold $(\bar{M},J, g)$, then $g((\nabla_XT)Y,Z)=g(Y,(\nabla_XT)Z)$, for any $X$, $Y$, $Z\in \Gamma(TM)$.
\end{proposition}
\begin{proof}
$$(\bar{\nabla}_XJ)Y=\bar{\nabla}_XTY+\bar{\nabla}_XNY-J(\nabla_XY+\sum_{\alpha=1}^rh_{\alpha}(X,Y)N_{\alpha})=$$$$=
(\nabla_XT)Y-\sum_{\alpha=1}^rg(NY,N_{\alpha})A_{N_{\alpha}}X-\sum_{\alpha=1}^rh_{\alpha}(X,Y)tN_{\alpha}+\sum_{\alpha=1}^rh_{\alpha}(X,TY)N_{\alpha}+
$$$$+\sum_{\alpha=1}^rX(g(NY,N_{\alpha}))N_{\alpha}+\sum_{1\leq \alpha, \beta \leq r}g(NY,N_{\beta})\lambda_{\beta\alpha}(X)N_{\alpha}-
\sum_{\alpha=1}^rh_{\alpha}(X,Y)nN_{\alpha}-N(\nabla_XY).$$

Since $NY=\sum_{\alpha=1}^rg(NY,N_{\alpha})N_{\alpha}$ and $g(NY,N_{\alpha})=g(JY,N_{\alpha})=g(Y,JN_{\alpha})$, we get:
$$g((\bar{\nabla}_XJ)Y,Z)=g((\nabla_XT)Y,Z)-\sum_{\alpha=1}^rg(NY,N_{\alpha})g(A_{N_{\alpha}}X,Z)-\sum_{\alpha=1}^rh_{\alpha}(X,Y)g(tN_{\alpha},Z)=$$
$$=g((\nabla_XT)Y,Z)-\sum_{\alpha=1}^r[g(Y,JN_{\alpha})h_{\alpha}(X,Z)+g(JN_{\alpha},Z)h_{\alpha}(X,Y)]$$
and from $g((\bar{\nabla}_XJ)Y,Z)=g(Y,(\bar{\nabla}_XJ)Z)$, we obtain:
$$g((\nabla_XT)Y,Z)-g(Y,(\nabla_XT)Z)=g((\bar{\nabla}_XJ)Y,Z)-g(Y,(\bar{\nabla}_XJ)Z)=0.$$
\end{proof}

\begin{proposition}\label{p3}
If $M$ is an isometrically immersed submanifold of the locally metallic Riemannian manifold $(\bar{M},J, g)$, then:
\begin{enumerate}
  \item for any $X$, $Y\in \Gamma(TM)$, we have:
\begin{enumerate}
  \item $(\nabla_XT)Y=\sum_{\alpha=1}^rg(NY,N_{\alpha})A_{N_{\alpha}}X+\sum_{\alpha=1}^rh_{\alpha}(X,Y)tN_{\alpha}$,
  \item $\sum_{\alpha=1}^rg(NY,N_{\alpha})\nabla^{\perp}_XN_{\alpha}=N(\nabla_XY)+\sum_{\alpha=1}^rh_{\alpha}(X,Y)nN_{\alpha}-  \sum_{\alpha=1}^rh_{\alpha}(X,TY)N_{\alpha}-
      \sum_{\alpha=1}^rX(g(NY,N_{\alpha}))N_{\alpha}$;
\end{enumerate}
   \item for any $X\in \Gamma(TM)$, $U\in \Gamma(TM^{\perp})$, we have:
\begin{enumerate}
  \item $\nabla_XtU=\sum_{\alpha=1}^rg(nU,N_{\alpha})[A_{N_{\alpha}}X-T(A_{N_{\alpha}}X)]+\sum_{\alpha=1}^r[X(g(U,N_{\alpha}))+\\
      \sum_{\beta=1}^rg(U,N_{\beta})\lambda_{\beta\alpha}(X)]tN_{\alpha}$,
  \item $\sum_{\alpha=1}^rg(nU,N_{\alpha})\nabla^{\perp}_XN_{\alpha}=-\sum_{\alpha=1}^r[X(g(nU,N_{\alpha}))+h_{\alpha}(X,tU)]N_{\alpha}+\\
      \sum_{\alpha=1}^r[X(g(U,N_{\alpha}))+
      \sum_{\beta=1}^rg(U,N_{\beta})\lambda_{\beta\alpha}(X)]nN_{\alpha}- 
      \sum_{\alpha=1}^rg(nU,N_{\alpha})N(A_{N_{\alpha}}X)$.
\end{enumerate}
\end{enumerate}
\end{proposition}
\begin{proof}
\begin{enumerate}
 \item For $X$, $Y\in \Gamma(TM)$:
$$NY=\sum_{\alpha=1}^rg(NY,N_{\alpha})N_{\alpha}$$
$$\bar{\nabla}_XNY=\sum_{\alpha=1}^rX(g(NY,N_{\alpha}))N_{\alpha}-\sum_{\alpha=1}^rg(NY,N_{\alpha})A_{N_{\alpha}}X+
\sum_{1\leq \alpha,\beta\leq r}g(NY,N_{\alpha})\lambda_{\alpha\beta}(X)N_{\beta}$$
$$J(\bar{\nabla}_XY)=J(\nabla_XY)+\sum_{\alpha=1}^rh_{\alpha}(X,Y)(tN_{\alpha}+nN_{\alpha})=$$$$=T(\nabla_XY)+N(\nabla_XY)+
\sum_{\alpha=1}^rh_{\alpha}(X,Y)(tN_{\alpha}+nN_{\alpha}).$$

It follows:
$$0=(\bar{\nabla}_XJ)Y=\bar{\nabla}_XTY+\bar{\nabla}_XNY-J(\bar{\nabla}_XY)=$$$$=
[\nabla_XTY-\sum_{\alpha=1}^rg(NY,N_{\alpha})A_{N_{\alpha}}X-T(\nabla_XY)-\sum_{\alpha=1}^rh_{\alpha}(X,Y)tN_{\alpha}]+$$
$$+[\sum_{\alpha=1}^rh_{\alpha}(X,TY)N_{\alpha}+\sum_{\alpha=1}^rX(g(NY,N_{\alpha}))N_{\alpha}
+$$$$+\sum_{1\leq \alpha,\beta\leq r}g(NY,N_{\alpha})\lambda_{\alpha\beta}(X)N_{\beta}-N(\nabla_XY)-\sum_{\alpha=1}^rh_{\alpha}(X,Y)nN_{\alpha}].$$
 \item For $X\in \Gamma(TM)$, $U\in \Gamma(TM^{\perp})$:
$$U=\sum_{\alpha=1}^rg(U,N_{\alpha})N_{\alpha}, \ \ nU=\sum_{\alpha=1}^rg(nU,N_{\alpha})N_{\alpha}$$
$$\bar{\nabla}_XU=\sum_{\alpha=1}^rX(g(U,N_{\alpha}))N_{\alpha}-\sum_{\alpha=1}^rg(U,N_{\alpha})A_{N_{\alpha}}X+
\sum_{1\leq \alpha,\beta\leq r}g(U,N_{\alpha})\lambda_{\alpha\beta}(X)N_{\beta}$$
$$\bar{\nabla}_XnU=\sum_{\alpha=1}^rX(g(nU,N_{\alpha}))N_{\alpha}-\sum_{\alpha=1}^rg(nU,N_{\alpha})A_{N_{\alpha}}X+
\sum_{1\leq \alpha,\beta\leq r}g(nU,N_{\alpha})\lambda_{\alpha\beta}(X)N_{\beta}$$
$$J(\bar{\nabla}_XU)=\sum_{\alpha=1}^rX(g(U,N_{\alpha}))(tN_{\alpha}+nN_{\alpha})-\sum_{\alpha=1}^rg(U,N_{\alpha})(T(A_{N_{\alpha}}X)+N(A_{N_{\alpha}}X))+$$$$+
\sum_{1\leq \alpha,\beta\leq r}g(U,N_{\alpha})\lambda_{\alpha\beta}(X)(tN_{\beta}+nN_{\beta}).$$

It follows:
$$0=(\bar{\nabla}_XJ)U=\bar{\nabla}_XtU+\bar{\nabla}_XnU-J(\bar{\nabla}_XU)=$$$$=
      [\nabla_XtU-\sum_{\alpha=1}^rg(nU,N_{\alpha})A_{N_{\alpha}}X-     \sum_{\alpha=1}^rX(g(U,N_{\alpha}))tN_{\alpha}+$$$$+\sum_{\alpha=1}^rg(U,N_{\alpha})T(A_{N_{\alpha}}X)-
\sum_{1\leq \alpha,\beta\leq r}g(U,N_{\alpha})\lambda_{\alpha\beta}(X)tN_{\beta}]+$$$$+[\sum_{\alpha=1}^rh_{\alpha}(X,tU)N_{\alpha}+
      \sum_{\alpha=1}^rX(g(nU,N_{\alpha}))N_{\alpha}+\sum_{1\leq \alpha,\beta\leq r}g(nU,N_{\alpha})\lambda_{\alpha\beta}(X)N_{\beta}-$$$$-
      \sum_{\alpha=1}^rX(g(U,N_{\alpha}))nN_{\alpha}+\sum_{\alpha=1}^rg(U,N_{\alpha})N(A_{N_{\alpha}}X)-
\sum_{1\leq \alpha,\beta\leq r}g(U,N_{\alpha})\lambda_{\alpha\beta}(X)nN_{\beta}].$$
    \end{enumerate}
\end{proof}

\begin{remark}
From the previous computations, we obtain for any $X$, $Y\in \Gamma(TM)$:
\begin{equation}\label{e4}
(\bar{\nabla}_XJ)Y=[\nabla_XTY-\sum_{\alpha=1}^rg(NY,N_{\alpha})A_{N_{\alpha}}X-T(\nabla_XY)-\sum_{\alpha=1}^rh_{\alpha}(X,Y)tN_{\alpha}]+
\end{equation}
$$+[\sum_{\alpha=1}^rh_{\alpha}(X,TY)N_{\alpha}+\sum_{\alpha=1}^rX(g(NY,N_{\alpha}))N_{\alpha}
+\sum_{1\leq \alpha,\beta\leq r}g(NY,N_{\alpha})\lambda_{\alpha\beta}(X)N_{\beta}-$$$$-N(\nabla_XY)-\sum_{\alpha=1}^rh_{\alpha}(X,Y)nN_{\alpha}].
$$
\end{remark}

Remark that, for $X\in T_xM$, the vector fields $JX$ and $JN_{\alpha}$ decompose into the tangential and the normal components:
\begin{equation}\label{e2}
JX=TX+\sum_{\alpha=1}^r\eta_{\alpha}(X)N_{\alpha},
\end{equation}
\begin{equation}\label{e3}
JN_{\alpha}=\xi_{\alpha}+\sum_{\beta=1}^ra_{\alpha\beta}N_{\beta},
\end{equation}
where $T$ is a tensor field of $(1,1)$-type on $M$ (which associates to tangent vector field $X$ on $M$ the tangential part of $JX$), $\xi_{\alpha}$ are vector fields and $\eta_{\alpha}$ are $1$-forms on $M$ ($1\leq \alpha \leq r$) and $(a_{\alpha\beta})_{1\leq \alpha,\beta\leq r}$ is an $r\times r$ matrix of smooth real functions on $M$, whose properties are stated in the next Proposition:

\begin{proposition}\cite{c}\label{pc}
If $M$ is an isometrically immersed $n$-dimensional submanifold of codimension $r$ of the $(n+r)$-dimensional metallic Riemannian manifold $(\bar{M},J, g)$, then the induced structure $\Sigma:=(T,g,\eta_{\alpha},\xi_{\alpha},(a_{\alpha\beta}))_{1\leq \alpha,\beta\leq r}$ on $M$ satisfies:
\begin{enumerate}
  \item $T^2=pT+qI_{\Gamma(TM)}-\sum_{\alpha=1}^r\eta_{\alpha}\otimes \xi_{\alpha}$;
  \item $\eta_{\alpha}\circ T=p\eta_{\alpha}-\sum_{\beta=1}^ra_{\alpha\beta}\eta_{\beta}$;
  \item $a_{\alpha\beta}=a_{\beta\alpha}$;
  \item $\eta_{\beta}(\xi_{\alpha})=q\delta_{\alpha\beta}+pa_{\alpha\beta}-\sum_{\gamma=1}^ra_{\alpha\gamma}a_{\gamma\beta}$;
  \item $T\xi_{\alpha}=p\xi_{\alpha}-\sum_{\beta=1}^ra_{\alpha\beta}\xi_{\beta}$;
  \item $\eta_{\alpha}(X)=i_{\xi_{\alpha}}g(X)$;
  \item $g(TX,Y)=g(X,TY)$;
  \item $g(TX,TY)=pg(X,TY)+qg(X,Y)-\sum_{\alpha=1}^r\eta_{\alpha}(X)\eta_{\alpha}(Y)$,
  \end{enumerate}
for any $X$, $Y\in \Gamma(TM)$, $1\leq \alpha,\beta\leq r$, where $\delta_{\alpha\beta}$ is the Kronecker delta.
\end{proposition}

Notice that in our notations, $\eta_{\alpha}(X)=g(NX,N_{\alpha})$, for $X\in \Gamma(TM)$ and $\xi_{\alpha}=tN_{\alpha}$ and $a_{\alpha\beta}=g(nN_{\alpha},N_{\beta})$.

\begin{proposition}\label{t}
Let $M$ be an isometrically immersed $n$-dimensional submanifold of codimension $r$ of the $(n+r)$-dimensional locally metallic Riemannian manifold $(\bar{M},J, g)$ and $\Sigma:=(T,g,\eta_{\alpha},\xi_{\alpha},(a_{\alpha\beta}))_{1\leq \alpha,\beta\leq r}$ is the induced structure on $M$. Then for any $X$, $Y\in \Gamma(TM)$ and $1\leq \alpha,\beta\leq r$:
\begin{enumerate}
  \item $(\nabla_XT)Y=\sum_{\alpha=1}^r[\eta_{\alpha}(Y)A_{N_{\alpha}}X+h_{\alpha}(X,Y)\xi_{\alpha}]$;
  \item $(\nabla_X\eta_{\alpha})Y=-h_{\alpha}(X,TY)+\sum_{\beta=1}^r[a_{\alpha\beta}h_{\beta}(X,Y)+\eta_{\beta}(Y)\lambda_{\alpha\beta}(X)]$;
  \item
     $\nabla_X\xi_{\alpha}=-T(A_{N_{\alpha}}X)+\sum_{\beta=1}^r[a_{\alpha\beta}A_{N_{\beta}}X+\lambda_{\alpha\beta}(X)\xi_{\beta}]$;
  \item $X(a_{\alpha\beta})=-[h_{\alpha}(X,\xi_{\beta})+h_{\beta}(X,\xi_{\alpha})]-
\sum_{\gamma=1}^r[a_{\alpha\gamma}\lambda_{\gamma\beta}(X)+a_{\beta\gamma}\lambda_{\gamma\alpha}(X)]$.
\end{enumerate}
\end{proposition}
\begin{proof}
$$0=(\bar{\nabla}_{X}J)Y:=\bar{\nabla}_{X}JY-J(\bar{\nabla}_{X}Y)=$$
$$=\bar{\nabla}_{X}J(TY)+\sum_{\alpha=1}^r\eta_{\alpha}(Y)\bar{\nabla}_{X}N_{\alpha}+\sum_{\alpha=1}^rX(\eta_{\alpha}(Y))N_{\alpha}
-J(\nabla_XY)-\sum_{\alpha=1}^rh_{\alpha}(X,Y)JN_{\alpha}=$$
$$=(\nabla_{X}T)Y+\sum_{\alpha=1}^rh_{\alpha}(X,TY)N_{\alpha}-\sum_{\alpha=1}^r\eta_{\alpha}(Y)A_{N_{\alpha}}X+
\sum_{\alpha=1}^rX(\eta_{\alpha}(Y))N_{\alpha}-$$$$-\sum_{\alpha=1}^r\eta_{\alpha}(\nabla_XY)N_{\alpha}-\sum_{\alpha=1}^rh_{\alpha}(X,Y)\xi_{\alpha}
+\sum_{1\leq \alpha,\beta \leq r}\lambda_{\alpha\beta}(X)\eta_{\alpha}(Y)N_{\beta}-\sum_{1\leq \alpha,\beta \leq r}a_{\alpha\beta}h_{\alpha}(X,Y)N_{\beta}=$$
$$=(\nabla_{X}T)Y-\sum_{\alpha=1}^r\eta_{\alpha}(Y)A_{N_{\alpha}}X-\sum_{\alpha=1}^rh_{\alpha}(X,Y)\xi_{\alpha}
+$$$$+\sum_{\alpha=1}^r[(\nabla_X\eta_{\alpha})Y+h_{\alpha}(X,TY)+\sum_{\beta=1}^r\lambda_{\alpha\beta}(X)\eta_{\beta}(Y)-
\sum_{\beta=1}^ra_{\alpha\beta}h_{\beta}(X,Y)]N_{\alpha},$$
from where we deduce that the tangential and the normal components should both vanish.

Also:
$$(\nabla_X\eta_{\alpha})Y:=X(\eta_{\alpha}(Y))-\eta_{\alpha}(\nabla_XY)=X(g(\xi_{\alpha},Y))-g(\xi_{\alpha}, \nabla_XY)=g(\nabla_X\xi_{\alpha},Y)$$
and from the second relation we get:
$$\nabla_X\xi_{\alpha}=-T(A_{N_{\alpha}}X)+\sum_{\beta=1}^r[a_{\alpha\beta}A_{N_{\beta}}X-g(\nabla^{\perp}_X{N_{\beta}},N_{\alpha})\xi_{\beta}].$$

For the last relation we have:
$$X(a_{\alpha\beta})=X(g(JN_{\alpha},N_{\beta}))=g(\bar{\nabla}_XJN_{\alpha},N_{\beta})+g(JN_{\alpha},\bar{\nabla}_XN_{\beta})=$$
$$=g((\bar{\nabla}_XJ)N_{\alpha},N_{\beta})+g(\bar{\nabla}_XN_{\alpha},JN_{\beta})+g(JN_{\alpha},\bar{\nabla}_XN_{\beta})$$
and replacing $\bar{\nabla}_XN_{\alpha}$ and $\bar{\nabla}_XN_{\beta}$ from (\ref{e5}) we get the required relation.
\end{proof}

From Proposition \ref{p3}, if $(\bar{M},J, g)$ is a locally Riemannian metallic manifold, then we can express the Nijenhuis tensor field of $T$ as follows.
\begin{proposition}\cite{hrbl} \label{p7}
If $M$ is an isometrically immersed $n$-dimensional submanifold of codimension $r$ of the $(n+r)$-dimensional locally metallic Riemannian manifold $(\bar{M},J, g)$ and $\Sigma:=(T,g,\eta_{\alpha},\xi_{\alpha},(a_{\alpha\beta}))_{1\leq \alpha,\beta\leq r}$ is the induced structure on $M$, then:
\begin{equation}
N_T(X,Y)=-\sum_{\alpha=1}^rg((TA_{N_{\alpha}}-A_{N_{\alpha}}T)X,Y)\xi_{\alpha}-\sum_{\alpha=1}^r\eta_{\alpha}(Y)(TA_{N_{\alpha}}-A_{N_{\alpha}}T)X+
\end{equation}$$+
\sum_{\alpha=1}^r\eta_{\alpha}(X)(TA_{N_{\alpha}}-A_{N_{\alpha}}T)Y,$$
for any $X$, $Y\in \Gamma(TM)$.
\end{proposition}

Following \cite{IMihai}, we can compute the components
$N^{(1)}$, $N^{(2)}$, $N^{(3)}$ and $N^{(4)}$ of the Nijenhuis
tensor field of $T$ for the induced structure $\Sigma:=(T,g,\eta_{\alpha},\xi_{\alpha},(a_{\alpha\beta}))_{1\leq \alpha,\beta\leq r}$ on the $n$-dimensional submanifold of codimension $r$ of the $(n+r)$-dimensional metallic Riemannian manifold $(\bar{M},J, g)$:
\begin{enumerate}
  \item $N^{(1)}(X,Y)=N_{T}(X,Y)-2\sum_{\alpha=1}^{r}d \eta_{\alpha}(X,Y)\xi_{\alpha}$;
  \item $N_{\alpha}^{(2)}(X,Y)=(\mathcal{L}_{TX}\eta_{\alpha})Y-(\mathcal{L}_{TY}\eta_{\alpha})X$;
  \item $N_{\alpha}^{(3)}(X)=(\mathcal{L}_{\xi_{\alpha}}T)X$;
  \item $N_{\alpha\beta}^{(4)}(X)=(\mathcal{L}_{\xi_{\alpha}}\eta_{\beta})X$,
\end{enumerate}
for any $X$, $Y \in \Gamma(TM)$ and $1\leq \alpha, \beta \leq r$, where $N_{T}$ is the Nijenhuis tensor field of $T$ and
$\mathcal{L}_{X}$ denoted the Lie derivative with respect to $X$.

\subsection{$\Sigma$-structures induced on submanifolds}

Let $\bar{M}$ be an isometrically immersed $(n+r)$-dimensional submanifold of codimension $1$ of the $(n+r+1)$-dimensional metallic Riemannian manifold $(\bar{\bar{M}},J, g)$, $M$ an isometrically immersed $n$-dimensional submanifold of codimension $r$ of $\bar{M}$ and denote also by $g$ the induced Riemannian metrics on $\bar{M}$ and $M$. Let $N$ be a unit vector field on $\bar{\bar{M}}$ normal to $\bar{M}$ and $\{N_1,\dots,N_r\}$ be an orthonormal basis for the normal space to $M$ in $\bar{M}$. Also, $\{N,N_1,\dots,N_r\}$ is an orthonormal basis for the normal space to $M$ in $\bar{\bar{M}}$. Thus, $M$ is an $n$-dimensional submanifold of codimension $r+1$ of $\bar{\bar{M}}$ and we have the following immersions between the Riemannian manifolds:
$$(M, g)\hookrightarrow (\bar{M}, g) \hookrightarrow (\bar{\bar{M}}, g).$$

Then for any $X\in \Gamma(T\bar{M})$ we have:
\begin{equation}
JX=\bar{T}X+\bar{\eta}(X)N,
\end{equation}
\begin{equation}
JN=\bar{\xi}+\bar{a}N,
\end{equation}
where $\bar{T}$ is a tensor field of $(1,1)$-type on $\bar{M}$ (which associates to tangent vector field $X$ on $\bar{M}$ the tangential part of $JX$), $\bar{\xi}$ is a vector field and $\bar{\eta}$ is a $1$-form on $\bar{M}$ and $\bar{a}$ is a smooth real function on $\bar{M}$, and for any $X\in \Gamma(TM)$ we have:
\begin{equation}
JX=TX+\sum_{\alpha=1}^r\eta_{\alpha}(X)N_{\alpha}+\eta(X)N,
\end{equation}
\begin{equation}
JN_{\alpha}=\xi_{\alpha}+\sum_{\beta=1}^ra_{\alpha\beta}N_{\beta}+aN,
\end{equation}
\begin{equation}
JN=\xi+\sum_{\alpha=1}^rb_{\alpha}N_{\alpha}+aN,
\end{equation}
where $T$ is a tensor field of $(1,1)$-type on $M$ (which associates to tangent vector field $X$ on $M$ the tangential part of $JX$), $\xi_{\alpha}$, $\xi$ are vector fields and $\eta_{\alpha}$, $\eta$ are $1$-forms on $M$ ($1\leq \alpha \leq r$), $(a_{\alpha\beta})_{1\leq \alpha,\beta\leq r}$ is an $r\times r$ matrix of smooth real functions on $M$ and $a=\bar{a}|_{M}$.

Let $\mathcal{A}:=\begin{pmatrix}
                    (a_{\alpha\beta})_{1\leq \alpha,\beta\leq r} & ^tA \\
                    A & a \\
                  \end{pmatrix}
$, where $A:=(g(JN_1,N),\dots,g(JN_r,N))$.

\begin{lemma}
The structure $\bar{\Sigma}:=(\bar{T},g,\bar{\eta},\bar{\xi},\bar{a})$ induced on the submanifold $(\bar{M},g)$ of codimension $1$ of the $(n+r+1)$-dimensional metallic Riemannian manifold $(\bar{\bar{M}},J, g)$ also induces on the submanifold $(M,g)$ of codimension $r$ of $\bar{M}$ a structure $\Sigma:=(T,g,\eta_{\alpha},\eta,\xi_{\alpha},\xi,\mathcal{A})_{1\leq \alpha\leq r}$ which has the following properties:
\begin{enumerate}
  \item $T^2=pT+qI_{\Gamma(TM)}-\sum_{\alpha=1}^r\eta_{\alpha}\otimes \xi_{\alpha}-\eta\otimes \xi$;
  \item $\eta_{\alpha}\circ T=p\eta_{\alpha}-\sum_{\beta=1}^ra_{\alpha\beta}\eta_{\beta}-b_{\alpha}\eta$;
  \item $\eta\circ T=(p-a)\eta$;
  \item $a_{\alpha\beta}=a_{\beta\alpha}$;
  \item $\eta_{\beta}(\xi_{\alpha})=q\delta_{\alpha\beta}+pa_{\alpha\beta}+b_{\beta}-\sum_{\gamma=1}^r(a_{\alpha\gamma}+
      b_{\gamma})(a_{\gamma\beta}+b_{\gamma})$;
  \item $\eta_{\alpha}(\xi)=(p-a)b_{\alpha}-\sum_{\beta=1}^ra_{\alpha\beta}b_{\beta}$;
  \item $\eta(\xi)=q+pa-a^2$;
  \item $T\xi_{\alpha}=p\xi_{\alpha}-\sum_{\beta=1}^ra_{\alpha\beta}\xi_{\beta}+(1-b_{\alpha})\xi$;
  \item $T\xi=(p-a)\xi-\sum_{\beta=1}^rb_{\beta}\xi_{\beta}$;
  \item $\eta_{\alpha}(X)=i_{\xi_{\alpha}}g(X)$;
  \item $\eta(X)=i_{\xi}g(X)$;
  \item $g(TX,Y)=g(X,TY)$;
  \item $g(TX,TY)=pg(X,TY)+qg(X,Y)-\sum_{\alpha=1}^r\eta_{\alpha}(X)\eta_{\alpha}(Y)-\eta(X)\eta(Y)$,
  \end{enumerate}
for any $X$, $Y\in \Gamma(TM)$, $1\leq \alpha,\beta\leq r$, where $\delta_{\alpha\beta}$ is the Kronecker delta.

Also:
\begin{enumerate}
  \item $\bar{T}^2=p\bar{T}+qI_{\Gamma(T\bar{M})}-\bar{\eta}\otimes \bar{\xi}$;
  \item $\bar{\eta}\circ \bar{T}=(p-\bar{a})\bar{\eta}$;
  \item $\bar{\eta}(\bar{\xi})=q+p\bar{a}-\bar{a}^2$;
  \item $\bar{T}\bar{\xi}=(p-\bar{a})\bar{\xi}$;
  \item $\bar{\eta}(X)=i_{\bar{\xi}}g(X)$;
  \item $g(\bar{T}X,Y)=g(X,\bar{T}Y)$;
  \item $g(\bar{T}X,\bar{T}Y)=pg(X,\bar{T}Y)+qg(X,Y)-\bar{\eta}(X)\bar{\eta}(Y)$,
  \end{enumerate}
for any $X$, $Y\in \Gamma(T\bar{M})$.
\end{lemma}

From the above considerations, we obtain that if $M$ is an isometrically immersed $n$-dimensional submanifold of codimension $r$ of the $(n+r)$-dimensional Riemannian manifold $(\bar{M},\bar{g})$ which is an isometrically immersed $(n+r)$-di\-men\-sio\-nal submanifold of codimension $1$ of the $(n+r+1)$-dimensional metallic Riemannian manifold $(\bar{\bar{M}},J, \bar{\bar{g}})$, then the induced structure $\Sigma:=(T,g,\eta_{\alpha},\eta,\xi_{\alpha},\xi,\mathcal{A})_{1\leq \alpha\leq r}$ on $M$ by the metallic Riemannian structure $(J,\bar{\bar{g}})$ on $\bar{\bar{M}}$ is the same with the one induced on $M$ by the structure $\bar{\Sigma}:=(\bar{T},\bar{g},\bar{\eta},\bar{\xi},\bar{a})$ induced on $\bar{M}$ by the metallic Riemannian structure $(J,\bar{\bar{g}})$ on $\bar{\bar{M}}$.

\medskip

Let $M_r$ be an $n$-dimensional submanifold of codimension $r$ ($r\geq 2$) of the $(n+r)$-dimensional metallic Riemannian manifold $(\bar{M},J,g)$. We make the following notations: $\bar{M}:=M_0$, $g:=g_0$, $J:=T_0$, so we have the sequence of Riemannian immersions:
$$(M_r, g_r)\hookrightarrow (M_{r-1},g_{r-1}) \hookrightarrow \dots \hookrightarrow (M_1, g_1) \hookrightarrow (\bar{M}, g):=(M_0,g_0),$$
where $g_i$ is the induced metric on $M_i$ by the metric $g_{i-1}$ on $M_{i-1}$, $1\leq i\leq r$ and each of the manifolds $(M_i,g_i)$ is an isometrically immersed submanifold of codimension $1$ of the Riemannian manifold $(M_{i-1},g_{i-1})$, $1\leq i\leq r$. Let $1\leq \alpha_i,\beta_i\leq i$ for any $1\leq i\leq r$. In this setting, we obtain:
\begin{proposition}
The structure $\Sigma_r:=(T_r,g_r,\eta_{\alpha_r}^r,\xi_{\alpha_r}^r,\mathcal{A}_r)_{1\leq \alpha_r\leq r}$ on the $n$-dimensional submanifold $M_r$ of codimension $r$ ($r\geq 2$) of the metallic Riemannian manifold $(\bar{M},g,J)$ induced on $M_r$ by the metallic Riemannian structure $(J,g)$ is the same with the one induced on $M_r$ by the structure $\Sigma_i:=(T_i,g_i,\eta_{\alpha_i}^i,\xi_{\alpha_i}^i,\mathcal{A}_i)_{1\leq \alpha_i\leq i}$ (with $i<r$, $1\leq \alpha_i\leq i$) induced on $M_i$ by the metallic Riemannian structure $(J,g)$, where $T_r$ is the tangential component of $T_i$ on $M_r$, the vector fields $\xi_{\alpha_r}^r$ are the tangential components on $M_r$ of the vector fields $\xi_{\alpha_i}^i$ on $M_i$, the $1$-forms $\eta_{\alpha_r}^r$ are the restrictions on $M_r$ of the $1$-forms $\eta_{\alpha_i}^i$ on $M_i$ and the matrix $\mathcal{A}_r:=(a_{\alpha_r\beta_r})_{1\leq \alpha_r,\beta_r\leq r}$ is defined by $a_{\alpha_r\beta_r}=a_{\beta_r\alpha_r}=g_r(T_{r-1}N_{\alpha_r},N_{\beta_r})$.
\end{proposition}

\subsection{Invariant and anti-invariant submanifolds}

\subsubsection{Invariant submanifolds}

A submanifold $M$ of $\bar{M}$ is called \textit{invariant} if $J(T_xM)\subset T_xM$, for any $x\in M$.\\
It follows $J(T_xM^{\perp})\subset T_xM^{\perp}$, for any $x\in M$, because for any $U\in \Gamma(TM^{\perp})$, $g(X,JU)=g(JX, U)=0$, for any $X\in \Gamma(TM)$.

\begin{proposition} \cite{c}
An isometrically immersed submanifold $M$ of a metallic Riemannian manifold $(\bar{M},J, g)$ is invariant with respect to $J$ if and only if $(M,T,g)$ is metallic Riemannian manifold whenever $T$ is non-trivial.
\end{proposition}

\begin{proposition}\label{pi}
If $M$ is an isometrically immersed invariant submanifold of the locally metallic Riemannian manifold $(\bar{M},J, g)$, then for any $X$, $Y\in \Gamma(TM)$:
\begin{equation}
\nabla J=0,
\end{equation}
\begin{equation}
\sum_{\alpha=1}^rh_{\alpha}(X,JY)N_{\alpha}=\sum_{\alpha=1}^rh_{\alpha}(X,Y)JN_{\alpha}=\sum_{\alpha=1}^rh_{\alpha}(JX,Y)N_{\alpha},
\end{equation}
\begin{equation}
h_{\alpha}(JX,JY)=ph_{\alpha}(X,JY)+qh_{\alpha}(X,Y), \ {\textit \ for \ any} \ 1\leq \alpha\leq r.
\end{equation}
\end{proposition}
\begin{proof}
Let $X$, $Y\in \Gamma(TM)$. Then $\nabla_XY$, $JX$, $JY$, $J(\nabla_XY)\in \Gamma(TM)$ and $JN_{\alpha}\in \Gamma(TM^{\perp})$, for any $1\leq \alpha\leq r$. If $\bar{\nabla}J=0$, using Proposition \ref{p3}:
$$(\nabla_XJ)Y:=\nabla_XJY-J(\nabla_XY)=\nabla_XTY-T(\nabla_XY):=(\nabla_XT)Y=0.$$

Also:
$$\sum_{\alpha=1}^rh_{\alpha}(X,Y)nN_{\alpha}-\sum_{\alpha=1}^rh_{\alpha}(X,TY)N_{\alpha}=0$$
which implies:
$$\sum_{\alpha=1}^rh_{\alpha}(X,JY)N_{\alpha}=\sum_{\alpha=1}^rh_{\alpha}(X,Y)JN_{\alpha}=\sum_{\alpha=1}^rh_{\alpha}(Y,X)JN_{\alpha}=$$$$=
\sum_{\alpha=1}^rh_{\alpha}(Y,JX)N_{\alpha}=\sum_{\alpha=1}^rh_{\alpha}(JX,Y)N_{\alpha}$$
and
$$h_{\alpha}(JX,JY)=h_{\alpha}(X,J^2Y)=ph_{\alpha}(X,JY)+qh_{\alpha}(X,Y).$$
\end{proof}

\begin{remark}
If $M$ is an isometrically immersed invariant $n$-dimensional submanifold of codimension $r$ of the $(n+r)$-dimensional metallic Riemannian manifold $(\bar{M},J, g)$ and $\Sigma:=(T,g,\eta_{\alpha},\xi_{\alpha},(a_{\alpha\beta}))_{1\leq \alpha,\beta\leq r}$ is the induced structure on $M$, then $\xi_{\alpha}$ are zero vector fields and the $1$-forms $\eta_{\alpha}$ vanish identically on $M$, for any $1\leq \alpha \leq r$. In this case, for any $X\in \Gamma(TM)$, (\ref{e2}) and (\ref{e3}) become:
\begin{equation}
JX=TX, \ \ JN_{\alpha}=\sum_{\beta=1}^ra_{\alpha\beta}N_{\beta}, \ {\textit \ for \ any} \ 1\leq \alpha\leq r.
\end{equation}

Also, the $\Sigma$-structure satisfies:
\begin{enumerate}
  \item $T^2=pT+qI_{\Gamma(TM)}$;
  \item $a_{\alpha\beta}=a_{\beta\alpha}$;
  \item $\sum_{\gamma=1}^ra_{\alpha\gamma}a_{\gamma\beta}=q\delta_{\alpha\beta}+pa_{\alpha\beta}$;
  \item $X(a_{\alpha\beta})=g((\bar{\nabla}_XJ)N_{\alpha},N_{\beta})-
\sum_{\gamma=1}^r[a_{\alpha\gamma}\lambda_{\gamma\beta}(X)+a_{\beta\gamma}\lambda_{\gamma\alpha}(X)]$;
  \item $g(TX,Y)=g(X,TY)$;
  \item $g(TX,TY)=pg(X,TY)+qg(X,Y)$,
  \end{enumerate}
for any $X$, $Y\in \Gamma(TM)$ and $1\leq \alpha,\beta\leq r$.
\end{remark}

Denoting by $\mathcal{J}:=\bar{\nabla}J$, from (\ref{e4}) we obtain:

\begin{proposition}\label{p8}
Let $M$ be an isometrically immersed invariant $n$-dimen\-sio\-nal submanifold of codimension $r$ of the $(n+r)$-dimensional metallic Riemannian manifold $(\bar{M},J, g)$ and $\Sigma:=(T,g,\eta_{\alpha}=0,\xi_{\alpha}=0,(a_{\alpha\beta}))_{1\leq \alpha,\beta\leq r}$ is the induced structure on $M$. Then:
\begin{enumerate}
  \item $\mathcal{J}(X,Y)^{T}=(\nabla_XT)Y$, \\
  $\mathcal{J}(X,Y)^{\perp}=
  \sum_{\alpha=1}^r[h_{\alpha}(X,TY)N_{\alpha}-h_{\alpha}(X,Y)nN_{\alpha}]$;
  \item $\mathcal{J}(X,N_{\alpha})^{T}=T(A_{N_{\alpha}}X)-\sum_{\beta=1}^ra_{\alpha\beta}A_{N_{\beta}}X$, \\
  $\mathcal{J}(X,N_{\alpha})^{\perp}=\sum_{\beta=1}^r[X(a_{\alpha\beta})+\sum_{\gamma=1}^r(a_{\alpha\gamma}\lambda_{\gamma\beta}(X)+
  a_{\beta\gamma}\lambda_{\gamma\alpha}(X))]N_{\beta}$,
  \end{enumerate}
for any $X$, $Y\in \Gamma(TM)$ and $1\leq \alpha \leq r$.
\end{proposition}

\pagebreak

From Proposition \ref{t} and Proposition \ref{p8} we deduce that:
\begin{corollary}
If $M$ is an isometrically immersed invariant $n$-dimen\-sio\-nal submanifold of codimension $r$ of the $(n+r)$-dimensional locally metallic Riemannian manifold $(\bar{M},J, g)$, $\Sigma:=(T,g,\eta_{\alpha}=0,\xi_{\alpha}=0,(a_{\alpha\beta}))_{1\leq \alpha,\beta\leq r}$ is the induced structure on $M$, then:
\begin{enumerate}
  \item $\nabla T=0$, \\
  $\sum_{\alpha=1}^r[h_{\alpha}(X,TY)N_{\alpha}-h_{\alpha}(X,Y)nN_{\alpha}]=0$;
  \item $T(A_{N_{\alpha}}X)-\sum_{\beta=1}^ra_{\alpha\beta}A_{N_{\beta}}X=0$, \\
  $X(a_{\alpha\beta})+\sum_{\gamma=1}^r[a_{\alpha\gamma}\lambda_{\gamma\beta}(X)+
  a_{\beta\gamma}\lambda_{\gamma\alpha}(X)]=0$,
  \end{enumerate}
for any $X$, $Y\in \Gamma(TM)$ and $1\leq \alpha, \beta \leq r$.
\end{corollary}

From $$d\eta_{\alpha}(X,Y)=-g((TA_{N_{\alpha}}-A_{N_{\alpha}}T)X,Y)+\sum_{\beta=1}^r[\lambda_{\alpha\beta}(X)\eta_{\beta}(Y)-
\lambda_{\alpha\beta}(Y)\eta_{\beta}(X)],$$ for any $X$, $Y\in \Gamma(TM)$, we obtain that:

\begin{proposition}
Let $M$ be an isometrically immersed invariant $n$-dimen\-sio\-nal submanifold of codimension $r$ of the $(n+r)$-dimensional locally metallic Riemannian manifold $(\bar{M},J, g)$ and $\Sigma:=(T,g,\eta_{\alpha}=0,\xi_{\alpha}=0,(a_{\alpha\beta}))_{1\leq \alpha,\beta\leq r}$ is the induced structure on $M$. Then
$$TA_{N_{\alpha}}=A_{N_{\alpha}}T,$$ for any $1\leq \alpha\leq r$ and the Nijenhuis tensor field of $T$ vanishes identically on $M$ (i.e. $N_{T}(X,Y)=0$, for any $X$, $Y \in \Gamma(TM)$).
\end{proposition}

\begin{proposition}
Let $M$ be an isometrically immersed invariant $n$-dimen\-sio\-nal submanifold of codimension $r$ of the $(n+r)$-dimensional metallic Riemannian manifold $(\bar{M},J, g)$ and $\Sigma:=(T,g,\eta_{\alpha}=0,\xi_{\alpha}=0,(a_{\alpha\beta}))_{1\leq \alpha,\beta\leq r}$ is the induced structure on $M$. Then the components $N^{(2)}$, $N^{(3)}$ and $N^{(4)}$ vanish identically on $M$. Moreover, if
$N_T=0$, then $N^{(1)}$ vanishes, too, on $M$. In particular, this happens if the normal connection $\nabla^{\perp}$ on the normal
bundle vanishes identically (i.e. $ \lambda_{\alpha\beta}=0$, for every
$1\leq \alpha, \beta \leq r$).
\end{proposition}
\pagebreak

\subsubsection{Anti-invariant submanifolds}

A submanifold $M$ of $\bar{M}$ is called \textit{anti-invariant} if $J(T_xM)\subset T_xM^{\perp}$, for any $x\in M$.

\begin{proposition}\label{pa}
If $M$ is an isometrically immersed anti-invariant submanifold of the locally metallic Riemannian manifold $(\bar{M},J, g)$, then for any $X$, $Y\in \Gamma(TM)$:
\begin{equation}
\sum_{\alpha=1}^rh_{\alpha}(X,Y)tN_{\alpha}=-\sum_{\alpha=1}^rg(JY,N_{\alpha})A_{N_{\alpha}}X,
\end{equation}
\begin{equation}
\sum_{\alpha=1}^rh_{\alpha}(X,Y)nN_{\alpha}=\sum_{\alpha=1}^rg(JY,N_{\alpha})\nabla_X^{\perp}N_{\alpha}+\sum_{\alpha=1}^rX(g(JY,N_{\alpha}))-J(\nabla_XY).
\end{equation}
\end{proposition}
\begin{proof}
Let $X$, $Y\in \Gamma(TM)$. Then $\nabla_XY\in \Gamma(TM)$, $J X$, $J{Y}$, $J(\nabla_XY)\in \Gamma(TM^{\perp})$. If $\bar{\nabla}J=0$, using Proposition \ref{p3}:
$$\sum_{\alpha=1}^rg(NY,N_{\alpha})A_{N_{\alpha}}X+\sum_{\alpha=1}^rh_{\alpha}(X,Y)tN_{\alpha}=0.$$

Also:
$$\sum_{\alpha=1}^rg(NY,N_{\alpha})\nabla^{\perp}_XN_{\alpha}=N(\nabla_XY)+\sum_{\alpha=1}^rh_{\alpha}(X,Y)nN_{\alpha}-
      \sum_{\alpha=1}^rX(g(NY,N_{\alpha}))N_{\alpha}.$$
\end{proof}

\begin{remark}
If $M$ is an isometrically immersed anti-invariant $n$-dimensional submanifold of codimension $r$ of the $(n+r)$-dimensional metallic Riemannian manifold $(\bar{M},J, g)$ and $\Sigma:=(T,g,\eta_{\alpha},\xi_{\alpha},(a_{\alpha\beta}))_{1\leq \alpha,\beta\leq r}$ is the induced structure on $M$, then $T$ vanishes identically on $M$. In this case, for any $X\in \Gamma(TM)$, (\ref{e2}) becomes:
\begin{equation}
JX=\sum_{\alpha=1}^r\eta_{\alpha}(X)N_{\alpha}.
\end{equation}

Also, the $\Sigma$-structure satisfies:
\begin{enumerate}
  \item $\sum_{\alpha=1}^r\eta_{\alpha}\otimes \xi_{\alpha}=qI_{\Gamma(TM)}$;
  \item $\sum_{\beta=1}^ra_{\alpha\beta}\eta_{\beta}(X)=p\eta_{\alpha}(X)$;
  \item $a_{\alpha\beta}=a_{\beta\alpha}$;
  \item $\eta_{\beta}(\xi_{\alpha})=q\delta_{\alpha\beta}+pa_{\alpha\beta}-\sum_{\gamma=1}^ra_{\alpha\gamma}a_{\gamma\beta}$;
  \item $\sum_{\beta=1}^ra_{\alpha\beta}\xi_{\beta}=p\xi_{\alpha}$;
  \item $X(a_{\alpha\beta})=g((\bar{\nabla}_XJ)N_{\alpha},N_{\beta})-[h_{\alpha}(X,\xi_{\beta})+h_{\beta}(X,\xi_{\alpha})]-
\sum_{\gamma=1}^r[a_{\alpha\gamma}\lambda_{\gamma\beta}(X)+a_{\beta\gamma}\lambda_{\gamma\alpha}(X)]$;
  \item $\eta_{\alpha}(X)=i_{\xi_{\alpha}}g(X)$,
 \end{enumerate}
for any $X$, $Y\in \Gamma(TM)$ and $1\leq \alpha,\beta\leq r$.
\end{remark}

Denoting by $\mathcal{J}:=\bar{\nabla}J$, from (\ref{e4}) we obtain:

\begin{proposition}\label{p9}
Let $M$ be an isometrically immersed anti-invariant $n$-dimensional submanifold of codimension $r$ of the $(n+r)$-dimensional metallic Riemannian manifold $(\bar{M},J, g)$ and $\Sigma:=(T=0,g,\eta_{\alpha},\xi_{\alpha},(a_{\alpha\beta}))_{1\leq \alpha,\beta\leq r}$ is the induced structure on $M$. Then:
\begin{enumerate}
  \item $\mathcal{J}(X,Y)^{T}=-\sum_{\alpha=1}^r\eta_{\alpha}(Y)A_{N_{\alpha}}X-\sum_{\alpha=1}^rh_{\alpha}(X,Y)\xi_{\alpha}$, \\ $\mathcal{J}(X,Y)^{\perp}=\sum_{\alpha=1}^rX(\eta_{\alpha}(Y))N_{\alpha}
-\sum_{1\leq \alpha,\beta\leq r}\lambda_{\alpha\beta}(X)\eta_{\beta}(Y)N_{\alpha}-\sum_{\alpha=1}^rh_{\alpha}(X,Y)nN_{\alpha}-N(\nabla_XY)$;
   \item $\mathcal{J}(X,N_{\alpha})^{T}=\nabla_X\xi_{\alpha}-\sum_{\beta=1}^ra_{\alpha\beta}A_{N_{\beta}}X-
      \sum_{\beta=1}^r\lambda_{\alpha\beta}(X)\xi_{\beta}$, \\
      $\mathcal{J}(X,N_{\alpha})^{\perp}=
      \sum_{\beta=1}^r[X(a_{\alpha\beta})+h_{\beta}(X,\xi_{\alpha})+
      \sum_{\gamma=1}^r(a_{\alpha\gamma}\lambda_{\gamma\beta}(X)+a_{\beta\gamma}\lambda_{\gamma\alpha}(X))]N_{\beta}+N(A_{N_{\alpha}}X)$,
  \end{enumerate}
for any $X$, $Y\in \Gamma(TM)$ and $1\leq \alpha\leq r$.
\end{proposition}

From Proposition \ref{t} and Proposition \ref{p9} we deduce that:
\begin{corollary}
If $M$ be an isometrically immersed anti-invariant $n$-dimen\-sio\-nal submanifold of codimension $r$ of the $(n+r)$-dimensional locally metallic Riemannian manifold $(\bar{M},J, g)$, $\Sigma:=(T=0,g,\eta_{\alpha},\xi_{\alpha},(a_{\alpha\beta}))_{1\leq \alpha,\beta\leq r}$ is the induced structure on $M$, then:
\begin{enumerate}
  \item $\sum_{\alpha=1}^r\eta_{\alpha}(Y)A_{N_{\alpha}}X+\sum_{\alpha=1}^rh_{\alpha}(X,Y)\xi_{\alpha}=0$, \\ $\sum_{\alpha=1}^rX(\eta_{\alpha}(Y))N_{\alpha}
-\sum_{1\leq \alpha,\beta\leq r}\lambda_{\alpha\beta}(X)\eta_{\beta}(Y)N_{\alpha}- \sum_{\alpha=1}^rh_{\alpha}(X,Y)nN_{\alpha}- N(\nabla_XY)=0$;
  \item $\nabla_X\xi_{\alpha}-\sum_{\beta=1}^ra_{\alpha\beta}A_{N_{\beta}}X-
      \sum_{\beta=1}^r\lambda_{\alpha\beta}(X)\xi_{\beta}=0$, \\
      $X(a_{\alpha\beta})+h_{\alpha}(X,\xi_{\beta})+h_{\beta}(X,\xi_{\alpha})+
      \sum_{\gamma=1}^r[a_{\alpha\gamma}\lambda_{\gamma\beta}(X)+a_{\beta\gamma}\lambda_{\gamma\alpha}(X)]=0$,\\
      $\sum_{\beta=1}^rh_{\alpha}(X,\xi_{\beta})N_{\beta}=N(A_{N_{\alpha}}X)$,
  \end{enumerate}
for any $X$, $Y\in \Gamma(TM)$ and $1\leq \alpha\leq r$.
\end{corollary}

\begin{proposition}
Let $M$ be an isometrically immersed anti-invariant $n$-dimensional submanifold of codimension $r$ of the $(n+r)$-dimensional metallic Riemannian manifold $(\bar{M},J, g)$ and $\Sigma:=(T=0,g,\eta_{\alpha},\xi_{\alpha},(a_{\alpha\beta}))_{1\leq \alpha,\beta\leq r}$ is the induced structure on $M$. Then the components $N^{(2)}$ and $N^{(3)}$ vanish identically on $M$. Moreover, if $\xi_{\alpha}$ are parallel with respect to a symmetric linear connection, for any $1\leq \alpha \leq r$, then $N^{(1)}$ and $N^{(4)}$ vanish, too, on $M$.
\end{proposition}

\subsection{Slant submanifolds}

The operator $T$ will essentially be involved in characterizing the slant submanifolds.

We say that $M$ is \textit{slant submanifold} if for any $x\in M$ and $X_x\in T_xM$, the angle $\theta(X_x)$ between $JX_x$ and $T_xM$ (which agrees with the angle between $JX_x$ and $TX_x$) is constant. In this case, we call $\theta=:\theta(X_x)$ the \textit{slant angle}. Invariant and anti-invariant submanifolds are particular cases of slant submanifolds with the slant angle $\theta=0$ and $\theta=\frac{\pi}{2}$, respectively. A slant immersion which is not invariant nor anti-invariant is called \textit{proper slant}.

For any nonzero tangent vector $X_x$ of $T_xM$, the cosine of the slant angle $\theta$ can be expressed as:
\begin{equation}
\cos(\theta(X_x)):=\frac{g(JX_x,TX_x)}{\sqrt{g(JX_x,JX_x)}\sqrt{g(TX_x,TX_x)}}=
\frac{||TX_x||}{||JX_x||}.
\end{equation}

Properties of slant submanifolds will be stated in the next Propositions.

\begin{proposition}\label{p5}
Let $M$ be an isometrically immersed submanifold of the metallic Riemannian manifold $(\bar{M},J, g)$. If $M$ is slant with the slant angle $\theta$, then:
\begin{equation}
g(TX,TY)=\cos^2(\theta)[pg(X,JY)+qg(X,Y)]
\end{equation}
and
\begin{equation}
g(NX,NY)=\sin^2(\theta)[pg(X,JY)+qg(X,Y)],
\end{equation}
for any $X$, $Y\in \Gamma(TM)$.
\end{proposition}
\begin{proof}
Taking $X+Y$ in $g(TX,TX)=\cos^2(\theta)g(JX,JX)$ we easily obtain $g(TX,TY)=\cos^2(\theta)g(JX,JY)=\cos^2(\theta)[pg(X,JY)+qg(X,Y)]$. Also,  $g(NX,NY)=g(JX,JY )-g(TX,TY)=\sin^2(\theta)[pg(X,JY)+qg(X,Y)]$.
\end{proof}

A characterization in terms of the $T$-operator of a slant submanifold of a metallic Riemannian manifold is now given:

\begin{proposition}\label{p4}
Let $M$ be an isometrically immersed submanifold of the metallic Riemannian manifold $(\bar{M},J, g)$. Then $M$ is slant if and only if there exists a real number $\lambda\in[0,1]$ such that
\begin{equation}
T^2=\lambda(pT+qI_{\Gamma(TM)}).
\end{equation}
\end{proposition}
\begin{proof}
If $M$ is slant, $\theta$ is constant and using the Proposition \ref{p5} we get:
$$g(T^2X,Y)=g(J(TX),Y)=g(TX,JY)=g(TX,TY)=$$$$=\cos^2(\theta)[pg(X,JY)+qg(X,Y)]=
\cos^2(\theta)g(pJX+qX,Y),$$
for any $X$, $Y\in \Gamma(TM)$.

Conversely, if there exists a real number $\lambda\in[0,1]$ such that $T^2=\lambda(pT+qI_{\Gamma(TM)})$ follows $\cos^2(\theta(X))=\lambda$, so $\lambda$ does not depend on $X$.
\end{proof}

\begin{proposition}
Let $M$ be an isometrically immersed submanifold of the metallic Riemannian manifold $(\bar{M},J, g)$. If $M$ is slant with the slant angle $\theta$, then:
\begin{equation}
(\nabla_XT^2)Y=p\cos^2(\theta)(\nabla_XT)Y,
\end{equation}
for any $X$, $Y\in \Gamma(TM)$.
\end{proposition}
\begin{proof}
We have:
$$T^2(\nabla_XY)=\cos^2(\theta)[pT(\nabla_XY)+q\nabla_XY]$$
and
$$\nabla_XT^2Y=\cos^2(\theta)[p\nabla_XTY+q\nabla_XY].$$

It follows:
$$(\nabla_XT^2)Y:=\nabla_XT^2Y-T^2(\nabla_XY)=p\cos^2(\theta)(\nabla_XT)Y.$$
\end{proof}

We deduce that:
\begin{corollary}
If $M$ is an isometrically immersed slant submanifold of the metallic Riemannian manifold $(\bar{M},J, g)$ with the slant angle $\theta$, then $\nabla T^2=0$ if and only if $M$ is anti-invariant or $(M,T,g)$ is locally metallic Riemannian manifold.
\end{corollary}

\begin{proposition}
Let $M$ be an isometrically immersed $n$-dimensional submanifold of codimension $r$ of the $(n+r)$-dimensional locally metallic Riemannian manifold $(\bar{M},J, g)$ and $\Sigma:=(T,g,\eta_{\alpha},\xi_{\alpha},(a_{\alpha\beta}))_{1\leq \alpha,\beta\leq r}$ is the induced structure on $M$. Then:
\begin{equation}
(\nabla_XT^2)Y=p\cos^2(\theta)\sum_{\alpha=1}^r[\eta_{\alpha}(Y)A_{N_{\alpha}}X+h_{\alpha}(X,Y)\xi_{\alpha}],
\end{equation}
for any $X$, $Y\in \Gamma(TM)$.

Moreover, if $T^2$ is Codazzi, then on $M$:
\begin{equation}
\sum_{\alpha=1}^r(\eta_{\alpha}\otimes A_{N_{\alpha}}-A_{N_{\alpha}}\otimes \eta_{\alpha})=0.
\end{equation}
\end{proposition}

From Proposition \ref{pc} and Proposition \ref{p4}, we can characterize $T^2$ in terms of the $\Sigma$-structure induced on $M$ as follows.
\begin{proposition}
Let $M$ be an isometrically immersed $n$-dimensional submanifold of codimension $r$ of the $(n+r)$-dimensional metallic Riemannian manifold $(\bar{M},J, g)$ and $\Sigma:=(T,g,\eta_{\alpha},\xi_{\alpha},(a_{\alpha\beta}))_{1\leq \alpha,\beta\leq r}$ is the induced structure on $M$. If $M$ is proper slant submanifold with the slant angle $\theta$, then:
\begin{equation}\label{e7}
T^2=\frac{1}{\tan^2 (\theta)}\sum_{\alpha=1}^r\eta_{\alpha}\otimes \xi_{\alpha}.
\end{equation}
\end{proposition}

\begin{remark}
Denote by $\mathcal{D}=\cap_{\alpha=1}^r \mathcal{D}_{\alpha}$ the intersection of the distributions $\mathcal{D}_{\alpha}:=\ker \eta_{\alpha}$, $1\leq \alpha\leq r$ and by $\mathcal{D}^{\perp}$ its orthogonal complement in $M$. Then:

i) $\mathcal{D}_{\alpha}$ is integrable if and only if $\eta_{\alpha}$ is closed;

ii) $T^2|_{\Gamma(\mathcal{D})}=0$ (from (\ref{e7}));

iii) $\xi_{\alpha}\in \mathcal{D}^{\perp}$, for any $1\leq \alpha\leq r$;

iv) the distribution $\mathcal{D}$ is $T$-invariant because
$$g(TX,\xi_{\alpha})=g(JX,\xi_{\alpha})=g(X,J\xi_{\alpha})=g(X,JN_{\alpha})-
\sum_{\beta=1}^ra_{\alpha\beta}g(X,N_{\beta})=$$$$=g(X,\xi_{\alpha})=0,$$
for any $X\in \Gamma(\mathcal{D})$;

v) if $(\bar{M},J, g)$ is a locally metallic Riemannian manifold, then for $X$, $Y\in \Gamma(\mathcal{D})$ we have $(\nabla_XT)Y=\sum_{\alpha=1}^rh_{\alpha}(X,Y)\xi_{\alpha}\in \Gamma(\mathcal{D}^{\perp})$ and $T$ is Codazzi.
\end{remark}

\section{Examples}

\begin{example}
Consider the Euclidean space $E^{2a+b}$ of dimension $(2a+b)$, $a$, $b \in \mathbb{N}^{*}$. For an almost product structure $F:E^{2a+b}\rightarrow E^{2a+b}$, $F(X^{i},Y^{i},Z^{j})=(X^{i},-Y^{i},Z^{j})$, let $J_{\lambda}:E^{2a+b}\rightarrow E^{2a+b}$ be the $(1,1)$-tensor field defined by:
\begin{equation}
J_{\lambda}(X^{i},Y^{i},Z^{j})=\frac{p}{2}(X^{i},Y^{i},Z^{j}) + \lambda \frac{\sqrt{\Delta}}{2}F(X^{i},Y^{i},Z^{j}),
\end{equation}
for $\lambda\in \{-1,1\}$, $(X^{i},Y^{i}, Z^j):=(X^{1},...,X^{a},Y^{1},...,Y^{a},Z^{1},...,Z^{b})$ from $E^{2a+b}$, $i \in \{1,...,a\}$, $j \in \{1,...,b\}$, where $\sigma=\sigma_{p,q}=\frac{p+\sqrt{\Delta}}{2}$ is a metallic number, $\overline{\sigma}=p-\sigma=\frac{p-\sqrt{\Delta}}{2}$ and $\Delta=p^{2}+4q$, for $p$ and $q$ positive integer numbers.

For $\lambda=1$:
\begin{equation}
J_{1}(X^{i},Y^{i},Z^{j})=(\sigma X^{i},\overline{\sigma} Y^{i}, \sigma Z^{j})
\end{equation}
and for $\lambda=-1$:
\begin{equation}
J_{-1}(X^{i},Y^{i},Z^{j})=(\overline{\sigma} X^{i},\sigma Y^{i}, \overline{\sigma} Z^{j}).
\end{equation}

$J_{1}$ is a metallic structure on $E^{2a+b}$ because:
$$J_{1}^{2}(X^{i},Y^{i},Z^{j})=(\sigma^{2} X^{i},\overline{\sigma}^{2} Y^{i}, \sigma^{2} Z^{j})=
((p\sigma+q) X^{i},(p\overline{\sigma}+q) Y^{i}, (p\sigma+q) Z^{j})=$$$$=p(\sigma X^{i},\overline{\sigma} Y^{i}, \sigma Z^{j})+q(X^{i},Y^{i},Z^{j})=
(pJ_1+qI)(X^{i},Y^{i},Z^{j}).$$

In the same manner we can find that $J_{-1}^{2}=pJ_{-1}+qI$, which implies that $J_{-1}$ is a metallic structure on $E^{2a+b}$.

In the following, we denote by $J_{\lambda}:=J$, where $\lambda \in \{-1,1\}$.

For any $(X^{i},Y^{i}, Z^j)$, $(X'^{i},Y'^{i}, Z'^j) \in \Gamma(TE^{2a+b})$, we have:
$$
\langle J(X^{i},Y^{i},Z^{j}) , (X'^{i},Y'^{i},Z'^{j})\rangle=\langle(X^{i},Y^{i},Z^{j}),J(X'^{i},Y'^{i},Z'^{j})\rangle.
$$

Thus the scalar product $\langle\cdot,\cdot\rangle$ on $E^{2a+b}$ is
$J$-compatible and $(E^{2a+b},\langle\cdot,\cdot\rangle,J)$ is a metallic Riemannian manifold.

Starting from the equation of the sphere $S^{2a+b-1}(R)$:
$$
\sum_{i=1}^{a}(x^{i})^{2}+\sum_{i=1}^{a}(y^{i})^{2}+\sum_{j=1}^{b}(z^{j})^{2}=R^{2},
$$
where $R$ is the radius of the sphere and $(x^{1},...,x^{a},y^{1},...,y^{a},z^{1},...,z^{b}):=(x^{i},y^{i},z^{j})$, $i \in \{1,...,a\}$, $j \in \{1,...,b\}$, are
the coordinates of a point of $S^{2a+b-1}(R)$, we use the
following notations:
$$
\sum_{i=1}^{a}(x^{i})^{2}=r_{1}^{2},\:
\sum_{i=1}^{a}(y^{i})^{2}=r_{2}^{2},\:
\sum_{j=1}^{b}(z^{j})^{2}=r_{3}^{2},\:
r_{1}^{2}+r_{2}^{2}:=r^{2}, \:
r^{2}+r_{3}^{2}=R^{2}.
$$

Remark that
$S^{2a-1}(r) \times S^{b-1}(r_{3})$ is a submanifold of
codimension 1 in $S^{2a+b-1}(R)$ and we have the successive
embeddings: $$S^{2a-1}(r) \times S^{b-1}(r_{3})
 \hookrightarrow S^{2a+b-1}(R) \hookrightarrow E^{2a+b}.$$

The tangent space in a point $(x^{i},y^{i},z^{j})$ at the product
of spheres $S^{2a-1}(r) \times S^{b-1}(r_{3})$ is:
$$
T_{(x^{1},...,x^{a},y^{1},...,y^{a},\underbrace{o,...,o}_{b})} S^{2a-1}(r)  \oplus  T_{(\underbrace{o,...,o}_{2a},z^{1},...,z^{b})}S^{b-1}(r_{3}).
$$

A vector $(X^{1},...,X^{a},Y^{1},...,Y^{a})$ from
$T_{(x^{1},...,x^{a},y^{1},...,y^{a})}E^{2a}$ is tangent to $S^{2a-1}(r)$ if and only if
$\sum_{i=1}^{a}(x^{i}X^{i}+y^{i}Y^{i})=0$ and it can be identified with
$(X^{1},...,X^{a},Y^{1},...,Y^{a},\underbrace{0,...,0}_{b})$ from $E^{2a+b}$.

A vector $(Z^{1},...,Z^{b})$ from
$T_{(z^{1},...,z^{b})}E^{b}$ is tangent to $S^{b-1}(r_{3})$ if and
only if $\sum_{j=1}^{b}z^{j}Z^{j}=0$ and it can be identified with $(\underbrace{0,...,0}_{2a},Z^{1},...,Z^{b})$ from $E^{2a+b}$.

Consequently, for any point $(x^{i},y^{i},z^{j}) \in S^{2a-1}(r) \times S^{b-1}(r_{3})$,
we have:
$$(X^{i},Y^{i},Z^{j}) \in T_{(x^{1},...,x^{a},y^{1},...,y^{a},z^{1},...,z^{b})}(S^{2a-1}(r) \times S^{b-1}(r_{3}))$$
if and only if
$
\sum_{i=1}^{a}(x^{i}X^{i}+y^{i}Y^{i})=0$ and $\sum_{j=1}^{b}z^{j}Z^{j}=0
$.

Using the unit normal vector fields on the sphere $S^{2a+b-1}(R)$ and $S^{2a-1}(r) \times S^{b-1}(r_{3})$, respectively, given by:
$$
N_{1}:= \frac{1}{R}(x^{i},y^{i},z^{j}), \:\:N_{2}=\frac{1}{R}\left(\frac{r_{3}}{r}x^{i},\frac{r_{3}}{r}y^{i},-\frac{r}{r_{3}}z^{j}\right),
$$
for any $i \in \{1,...,a\}$, $j \in \{1,...,b\}$, we obtain an orthonormal basis $\{N_{1},N_{2}\}$ of \linebreak
$
T_{(x^{i},y^{i},z^{j})}^{\perp}(S^{2a-1}(r) \times S^{b-1}(r_{3}))
$, for $N_{1}$ the unit vector field normal at $S^{2a+b-1}(R)$ and also, normal at $S^{2a-1}(r) \times S^{b-1}(r_{3})$.

For $J:=J_{1}$:
$$J N_{1}=\frac{1}{R}(\sigma x^{i},\overline{\sigma} y^{i}, \sigma z^{j}).$$

From the decomposition of $JN_{k}=\xi_{k}+a_{k1}N_{1}+a_{k2}N_{2}, \ \ k \in \{1,2\}$, into the tangential and normal components at $S^{2a+b-1}(R)$ and from $a_{kt}=\langle JN_{k},N_{t}\rangle$, $k$, $t \in \{1,2\}$, we
obtain:
\begin{equation}
a_{11}=\frac{\sigma r_{1}^{2}+\overline{\sigma}r_{2}^{2} + \sigma r_{3}^{2}}{R^{2}}, \:
a_{12}=a_{21}=-\frac{(\sigma-\overline{\sigma})r_{3}r_{2}^{2}}{r R^{2}}, \:
a_{22}=\frac{r_{3}(\sigma r_{1}^{2}+\overline{\sigma}r_{2}^{2} + \sigma r^{2})}{r R^{2}}
\end{equation}
and the matrix $\mathcal{A}:=(a_{\alpha\beta})_{2}$ is given by:
\begin{equation}\label{e11}
\mathcal{A}= \frac{1}{r R^{2}}\begin{pmatrix}
r(\sigma r_{1}^{2}+\overline{\sigma} r_{2}^{2} + \sigma r_{3}^{2}) & (\overline{\sigma}-\sigma)r_{3}r_{2}^{2} \\
   (\overline{\sigma}-\sigma)r_{3}r_{2}^{2} & r_{3}(\sigma r_{1}^{2}+\overline{\sigma}r_{2}^{2} + \sigma r^{2})
\end{pmatrix}.
\end{equation}

The tangential components $\xi_{k}=J N_{k}-a_{k1}N_{1}-a_{k2}N_{2}$ of $J N_{k}$, $k \in \{1,2\}$, at the sphere $S^{2a+b-1}(R)$ are given by:
\begin{equation}\label{e12}
\xi_{1}=\left(\frac{(\sigma-\overline{\sigma})r_{2}^{2}}{r^{2}} x^{i}, -\frac{(\sigma-\overline{\sigma})r_{1}^{2}}{r^{2}}y^{i},0\right)
\end{equation}
and
\begin{equation}\label{e13}
\xi_{2}=\frac{1}{rR^{3}}\left(\left(r_{3}\tau-\frac{p r_{2}^{2}r_{3}^{2}}{r}\right)x^{i},\left(r_{3}\tau-\frac{p r_{2}^{2}r_{3}^{2}}{r}\right)y^{i},\left(r \tau -\sqrt{\Delta} r_{2}^{2}r_{3}-\frac{rR^{2}\sigma}{r_{3}^{2}}\right)z^{j}\right),
\end{equation}
for $i \in \{1,...,a\}$ and $j \in \{1,...,b\}$, where $\tau = \sigma r_{1}^{2}+\overline{\sigma}r_{2}^{2}+\sigma r_{3}^{2}$.

If we decompose $J(X^{i},Y^{i},Z^{j})$ into the tangential and normal components
at the sphere $S^{2a+b-1}$, where $(X^{i},Y^{i},Z^{j})$ is a tangent vector field at
$S^{2a+b-1}(R)$, we obtain:
$$J (X^{i},Y^{i},Z^{j}) = T(X^{i},Y^{i},Z^{j})+ \eta_{1}(X^{i},Y^{i},Z^{j})N_{1} + \eta_{2}(X^{i},Y^{i},Z^{j})N_{2}.$$

From $\eta_{k}(X^{i},Y^{i},Z^{j})=\langle(X^{i},Y^{i},Z^{j}),\xi_{k}\rangle$, $k \in \{ 1,2\}$, we obtain:
\begin{equation}\label{e14}
\eta_{1}(X^{i},Y^{i},Z^{j})= \sqrt{\Delta}\sum_{i=1}^{a}x^{i}X^{i},\quad \eta_{2}(X^{i},Y^{i},Z^{j})=0.
\end{equation}

Moreover, we have:
\begin{equation}\label{e15}
 T(X^{i},Y^{i},Z^{j})= \left(\sigma X^{i}-\frac{\sqrt{\Delta}}{R}s x^{i},\overline{\sigma} Y^{i}-\frac{\sqrt{\Delta}}{R}s y^{i},\sigma Z^{j}-\frac{\sqrt{\Delta}}{R}s z^{j}\right),
 \end{equation}
where $s=\sum_{i=1}^{a}x^{i}X^{i}=-\sum_{i=1}^{a}y^{i}Y^{i}$ and $(X^{i},Y^{i},Z^{j})$ is a tangent vector at $S^{2a-1}(r) \times S^{b-1}(r_{3})$ in any
point $(x^{i},y^{i},z^{j})$.

Therefore, the induced structure $(T,\langle\cdot,\cdot\rangle,\xi_{1},\xi_{2}, \eta_{1},\eta_{2}, \mathcal{A})$ by the metallic structure $J:=J_{1}$ from $E^{2a+b}$ on the product of spheres $S^{2a-1}(r)\times S^{b-1}(r_3)$
in the Euclidean space $E^{2a+b}$ is given by the equations (\ref{e15}), (\ref{e12}), (\ref{e13}), (\ref{e14}) and (\ref{e11}). Similarly, we can find the induced
structure by the metallic structure $J:=J_{2}$ from $E^{2a+b}$ on the product of spheres $S^{2a-1}(r)\times S^{b-1}(r_3)$ in the Euclidean space $E^{2a+b}$.
\end{example}

\begin{example}
Like in the Golden case \cite{ch}, we construct an invariant submanifold of a metallic Riemannian manifold.

Consider the Euclidean space $E^{a+b}$ of dimension $(a+b)$, $a$, $b \in \mathbb{N}^{*}$. Let $J:E^{a+b}\rightarrow E^{a+b}$ be the $(1,1)$-tensor field defined by:
\begin{equation}\label{eq1}
J(X^{1},...,X^{a},Y^{1},...,Y^{b})=(\sigma X^{1},...,\sigma X^{a}, \overline{\sigma} Y^{1},...,\overline{\sigma} Y^{b}),
\end{equation}
for $(X^{1},...,X^{a},Y^{1},...,Y^{b}):=(X^{i},Y^{j})$ from $E^{a+b}$, $i \in \{1,...,a\}$, $j \in \{1,...,b\}$, where $\sigma=\sigma_{p,q}=\frac{p+\sqrt{\Delta}}{2}$ is a metallic number, $\overline{\sigma}=p-\sigma=\frac{p-\sqrt{\Delta}}{2}$ and $\Delta=p^{2}+4q$, for $p$ and $q$ positive integer numbers.

$J$ is a metallic structure on $E^{a+b}$ because:
$$
J^{2}(X^{i},Y^{j})=(\sigma^{2}X^{1},...,\sigma^{2}X^{a},\overline{\sigma}^{2}Y^{1},...,\overline{\sigma}^{2}Y^{b})=
$$
$$=p(\sigma X^{1},...,\sigma X^{a},\overline{\sigma}Y^{1},..., \overline{\sigma} Y^{b})+q(X^{1},...,X^{a},Y^{1},...,Y^{b})=(pJ+qI)(X^{i},Y^{j}).$$

For any $(X^{i},Y^{j})$, $(X'^{i},Y'^{j}) \in \Gamma(TE^{a+b})$, we have:
$$\langle J(X^{i},Y^{j}) , (X'^{i},Y'^{j})\rangle=\langle(X^{i},Y^{j}),J(X'^{i},Y'^{j})\rangle.$$

Thus, the scalar product $\langle\cdot,\cdot\rangle$ on $E^{a+b}$ is
$J$-compatible and $(E^{a+b},\langle\cdot,\cdot\rangle,J)$ is a metallic Riemannian manifold.

From $E^{a+b}=E^{a} \times E^{b}$, in each of the spaces $E^{a}$ and $E^{b}$ we can get the hyperspheres:
$$S^{a-1}(r_{1})=\{(x^{1},...,x^{a})\in E^{a},\sum_{i=1}^{a}(x^{i})^{2}=r_{1}^{2}\},$$ $$S^{b-1}(r_{2})=\{(y^{1},...,y^{b})\in E^{b},\sum_{j=1}^{b}(y^{j})^{2}=r_{2}^{2}\},$$
respectively. Thus, we can construct the product manifold $S^{a-1}(r_{1}) \times
S^{b-1}(r_{2})$.

Let $(x^{1},...,x^{a},y^{1},...,y^{b}):=(x^{i},y^{j})$, $i \in \{1,...,a\}$, $j \in \{1,...,b\}$, be the coordinates of a point of $S^{a-1}(r_{1}) \times S^{b-1}(r_{2})$ so that:
$$
\sum_{i=1}^{a}(x^{i})^{2}+\sum_{j=1}^{b}(y^{j})^{2}=r_{1}^{2}+r_{2}^{2}:=r^{2}.
$$

Remark that $S^{a-1}(r_{1}) \times S^{b-1}(r_{2})$ is a submanifold of codimension $2$ in $E^{a+b}$, $S^{a-1}(r_{1}) \times S^{b-1}(r_{2})$ is a submanifold of codimension $1$ in
$S^{a+b-1}(r)$ and we have the successive embeddings:
$$S^{a-1}(r_{1}) \times S^{b-1}(r_{2}) \hookrightarrow S^{a+b-1}(r) \hookrightarrow E^{a+b}.$$

The tangent space in a point $(x^{1},...,x^{a},y^{1},...,y^{b}):=(x^{i},y^{j})$ at the product
of spheres $S^{a-1}(r_{1}) \times S^{b-1}(r_{2})$ is:
$$T_{(x^{1},...,x^{a},\underbrace{o,...,o}_{b})} S^{a-1}(r_{1})  \oplus
 T_{(\underbrace{o,...,o}_{a},y^{1},...,y^{b})}S^{b-1}(r_{2}).$$

A vector $(X^{1},...,X^{a})$ from $T_{(x^{1},...,x^{a})}E^{a}$ (respectively $(Y^{1},...,Y^{b})$ from $T_{(y^{1},...,y^{b})}E^{b}$) is tangent to $S^{a-1}(r_{1})$ (respectively to $S^{b-1}(r_{2})$) if and only if
$\sum_{i=1}^{a}x^{i}X^{i}=0$ (respectively $\sum_{j=1}^{b}y^{j}Y^{j}=0$)
and it can be identified with
$(X^{1},...,X^{a},\underbrace{0,...,0}_{b})$ from $E^{a+b}$ (respectively with
 $(\underbrace{0,...,0}_{a},Y^{1},...,Y^{b})$ from $E^{a+b}$).

Consequently, for any point $(x^{i},y^{j}) \in S^{a-1}(r_{1}) \times S^{b-1}(r_{2})$, we have:
$$(X^{i},Y^{j}) \in T_{(x^{i},y^{j})}(S^{a-1}(r_{1}) \times S^{b-1}(r_{2})) \subset T_{(x^{i},y^{j})}S^{a+b-1}(r)$$
if and only if $\sum_{i=1}^{a}x^{i}X^{i}=\sum_{j=1}^{b}y^{j}Y^{j}=0$.

We consider a local orthonormal basis $\{N_{1},N_{2}\}$ of
$T_{(x^{i},y^{j})}^{\perp}(S^{a-1}(r_{1}) \times S^{b-1}(r_{2}))$ in
a point $(x^{i},y^{j}) \in S^{a-1}(r_{1}) \times S^{b-1}(r_{2})$, given by:
$$
N_{1}= \frac{1}{r}(x^{i},y^{j}), \ \ N_{2}=\frac{1}{r}\left(\frac{r_{2}}{r_{1}}x^{i},-\frac{r_{1}}{r_{2}}y^{j}\right),
$$
for any $i \in \{1,...,a\}$, $j \in \{1,...,b\}$.

From the decomposition of $JN_{k}=\xi_{k}+a_{k1}N_{1}+a_{k2}N_{2}, \ \ k \in \{1,2\}$, into the tangential and normal components at $S^{a-1}(r_{1})
\times S^{b-1}(r_{2})$ and from $a_{kt}=\langle JN_{k},N_{t}\rangle$, $k$, $t \in \{1,2\}$, we
obtain:
\begin{equation}\label{e17}
a_{11}=\frac{\sigma r_{1}^{2}+ \overline{\sigma} r_{2}^{2}}{r^{2}},
\quad a_{12}=a_{21}=\frac{r_{1}r_{2}(\sigma-\overline{\sigma})}{r^{2}}, \quad
a_{22}=\frac{\overline{\sigma}r_{1}^{2}+ \sigma r_{2}^{2}}{r^{2}}
\end{equation}
and the matrix $\mathcal{A}:=(a_{\alpha\beta})_{2}$ is given by:
\begin{equation}\label{eq5}
\mathcal{A}= \frac{1}{r^{2}}\begin{pmatrix}
(\sigma r_{1}^{2}+\overline{\sigma} r_{2}^{2}) & r_{1}r_{2}(\sigma-\overline{\sigma}) \\
   r_{1}r_{2}(\sigma-\overline{\sigma}) & \overline{\sigma}r_{1}^{2}+ \sigma r_{2}^{2}
\end{pmatrix}.
\end{equation}

We obtain:
\begin{equation}\label{eq6}
\xi_{1}= \xi_{2}= \underbrace{(0,...,0)}_{a+b}
\end{equation}
therefore:
$$
J (T_{x}^{\perp}(S^{a-1}(r_{1}) \times S^{b-1}(r_{2})))\subseteq T_{x}^{\perp}(S^{a-1}(r_{1}) \times S^{b-1}(r_{2})).
$$

From $\eta_{k}(X^{i},Y^{j})=\langle(X^{i},Y^{j}),\xi_{k}\rangle$, $k \in \{1,2\}$ and (\ref{eq6}), we obtain:
\begin{equation}\label{eq10}
\eta_{1}(X^{i},Y^{j})=\eta_{2}(X^{i},Y^{j})=0,
\end{equation}
for any tangent vector $(X^{i},Y^{j})$ on the product of spheres $S^{a-1}(r_{1}) \times S^{b-1}(r_{2})$ in a point
$(x^{i},y^{j}) \in S^{a-1}(r_{1}) \times S^{b-1}(r_{2})$.

From the decomposition of $J(X^{1},...,X^{a},Y^{1},...,Y^{b}):=J(X^{i},Y^{j})$ into the tangential and normal components at $S^{a-1}(r_{1}) \times S^{b-1}(r_{2})$ we obtain:
\begin{equation}\label{eq11}
T(X^{i},Y^{j}) = J(X^{i},Y^{j}).
\end{equation}

Thus, we have $J (T_{x}(S^{a-1}(r_{1}) \times S^{b-1}(r_{2})))\subseteq T_{x}(S^{a-1}(r_{1}) \times S^{b-1}(r_{2}))$ and
$$T^{2}(X^{i},Y^{j})=pT(X^{i},Y^{j})+q(X^{i},Y^{j})$$
and we obtain the induced structure
$(T,\langle\cdot,\cdot\rangle,0,0,0,0,\mathcal{A})$ on the product of spheres $S^{a-1}(r_{1}) \times S^{b-1}(r_{2})$ by the metallic structure $(J,\langle\cdot,\cdot\rangle)$ on $E^{a+b}$.
Also, $(T,\langle\cdot,\cdot\rangle)$ is a metallic Riemannian structure on $S^{a-1}(r_{1}) \times S^{b-1}(r_{2})$ and we remark that $S^{a-1}(r_{1}) \times S^{b-1}(r_{2})$ is an invariant manifold in the $(a+b)$-dimensional manifold $E^{a+b}$, $a$, $b \in \mathbb{N}^{*}$.
\end{example}

{\bf Adara M. Blaga} \\
West University of Timi\c{s}oara\\
Timi\c{s}oara, 300223, Romania\\
e-mail: adarablaga@yahoo.com

\medskip

{\bf Cristina E. Hre\c{t}canu}\\
Stefan cel Mare University of Suceava \\
Suceava, 720229, Romania\\
e-mail: criselenab@yahoo.com

\end{document}